\documentclass{birkjour}
 \newtheorem{thm}{Theorem}[section]
 \newtheorem{cor}[thm]{Corollary}

 \theoremstyle{definition}
 \newtheorem{defn}[thm]{Definition}
 \theoremstyle{remark}
 
 \newtheorem*{ex}{Example}
 \numberwithin{equation}{section}

\usepackage[utf8]{inputenc}
\usepackage{amsmath}
\usepackage{amssymb}
\usepackage{mathrsfs}
\usepackage{amsthm}
\usepackage{cite}
\usepackage{graphicx}
\usepackage{hyperref}
\usepackage{enumitem}
\usepackage{combelow}
\begin{document}

%
%
%
%
%
%
%
%
%

\title[Generalized $TAC$-contraction on Partial $b$-metric spaces]
 {Common Fixed Point Theorems with Generalized $TAC$-Contraction on Partial $b$-Metric Space}

\author[A. Gupta]{Anuradha Gupta }

\address{Department of Mathematics, Delhi College of Arts and Commerce, \\
University of Delhi, \\
Netaji Nagar, \\
New Delhi-110023, \\
India.}

\email{dishna2@yahoo.in}

\author[R. Mansotra]{Rahul Mansotra}
\address{Department of Mathematics\\University of Delhi\\
New Delhi 110023}
\email{mansotrarahul2@gmail.com}

\begin{abstract}
In this paper, we give common coincidence point 
 and common fixed point theorems for four self maps in the setting of generalized $TAC$-contraction in partial $b$-metric space. Also, we give  an example to authenticate the viability of the results. 
\end{abstract}

\maketitle
\section{Introduction and Preliminaries}
\begin{verbatim}\end{verbatim}
Results of fixed point theory is being constantly used in various mathematical disciplines, such as   general topology, operator theory, nonlinear analysis and many more. It has  applications in every aspect of life including various field of science. For instance,  engineering, physics, biology, economics, astrophysics etc. The pioneer result in this field is known as Banach Contraction Principle\cite{bm6}, which is considered as the origin of fixed point theory in metric spaces. Several variations of metric spaces has been developed by researchers in last two decades in the form of partial $b$-metric space, quasi-metric, 
 quasi $b$-metric, quasi partial $b$-metric space etc.
\par Alber and Guerre-Delabrere\cite{bm2} in 1997 gave a result which says a weakly  contractive map from Hilbert space to itself becomes a picard operator. Rhoades\cite{bm14} gave a generalized version of weakly contractive map on a complete metric space. ($\psi,\phi$)-weakly contractive map 
 in metric space was given by  Dutta  and Choudhury\cite{bm9}. In succession, Abbas and Doric\cite{bm1} gave common fixed point results on generalized ($\psi,\phi$)-weakly contractive map for four maps. Consequently, Chandok et.al\cite{bm7} came up with $TAC$-contraction and some fixed point results on them which is further extended to $b$-metric space by Babu and Dula\cite{bm5}.
 
 \par 
 Firstly, we give some preliminaries, definitions and notations, which will be used subsequently in the paper.
 \par The $b-$metric space was introduced by Czerwik\cite{bm8} as follows:
\begin{defn}\cite{bm8} Let $\mathcal{O}$ be a non empty set, $s \geq 1$ and $b :\mathcal{O} \times \mathcal{O} \rightarrow [ 0, \infty )$ is said to be  $b$-metric on $\mathcal{O}$ such that for all $z, s , t \in \mathcal{O}$:
\begin{enumerate}[label=(\roman*)]
\item$b(s,t)=0$ if and only if $  s = t $;
    \item $b(s,t) = b(t,s)$;
    \item $b(z,s) \leq s[b(z,t) + b(t,s)]$.
\end{enumerate}
Then, $(\mathcal{O},b,s)$ is called a $b$-metric space.
\end{defn}
 Recently, Shukla\cite{bm16} extended the definition of $b-$metric space as follows: 
\begin{defn} \cite{bm16} Let $\mathcal{O}$ be a non empty set, $s \geq 1$ and $\delta:\mathcal{O} \times \mathcal{O} \rightarrow [ 0, \infty )$ is said to be partial $b$-metric on $\mathcal{O}$ such that for all $z, s, t \in \mathcal{O} $:
\begin{enumerate}[label=(\roman*)]
    \item$ s = t$ if and only if $\delta(s,s) = \delta(s,t) = \delta(t,t)$;
    \item $\delta(z,z) \leq \delta(z,s)$;
    \item $\delta(z,s) = \delta(s,z)$;
    \item $\delta(z,s) \leq s[\delta(z,t) + \delta(t,s)] - \delta(t,t)$.
\end{enumerate}
Then, $( \mathcal{O} , \delta,s)$ is called a partial $b$-metric space. \end{defn}
\begin{defn}\cite{bm16}  Let $( \mathcal{O} , \delta , s)$ be a partial b-metric space. Then:
\begin{enumerate}[label=(\roman*)]
    \item  A sequence $ \{ z_{k}\}$ in $\mathcal{O}$ is said to be convergent if there is  $z \in \mathcal{O}$ such that  $ \lim_{k\rightarrow \infty} \delta(z_{k},z )=\delta(z,z).$
    \item A sequence  $ \{ z_{k}\}$  in $\mathcal{O}$ is said to be cauchy  in $ \mathcal{O}$ if  $ \lim_{k , m\rightarrow \infty} \delta(z_{k} , z_{m} )$  exists and is finite.
    \item $\mathcal{O}$ is said to be complete if for every Cauchy sequence $\{z_{k}\}$  there is $z \in \mathcal{O}$ such that $\lim_{k,m \rightarrow \infty } \delta(z_{k}  , z_{m}) =  \lim_{k \rightarrow \infty }\delta( z_{k}, z ) = \delta(z,z).$
    \item $A \subseteq \mathcal{O}$ is said to be partially $b$-closed if $\Bar{A} =A.$
     \end{enumerate}
     \end{defn}
The definition of cyclic 
 $( \gamma, \delta)$-
 admissible mapping for four self maps was given by  Hussain et.al.\cite{bm11} in the following way:
\begin{defn}\cite{bm11} Let $\mathcal{O}$ be a non empty set, $h, \eta , Q$ and $\mathcal{Z}:\mathcal{O}\rightarrow \mathcal{O}$ and $\gamma , \delta :\mathcal{O} \rightarrow [ 0, \infty ).$ Then, the pair $(h,\eta)$ is said to be cyclic $(\gamma , \delta)$-admissible with respect to $(Q,\mathcal{Z})$ if;
 \begin{enumerate}[label=(\roman*)]
     \item $\gamma(Qw) \geq 1$ for  $w \in \mathcal{O} $ implies $ \delta(hw) \geq 1$,
     \item  $\delta(\mathcal{Z}w) \geq 1$ for  $w \in \mathcal{O}$ implies $\gamma(\eta w) \geq 1$.
 \end{enumerate} \end{defn}
 \begin{defn}\cite{bm12}  Let $\mathcal{O}$ be a non empty set and $h, \eta$ be two self maps on $\mathcal{O}$. A point $w$ is said to be  coincidence point of pair $(h,\eta )$ if $hw = \eta w$.\end{defn}
Recently, Ansari\cite{bm4} gave  the definition of $C$-class function as follows:
\begin{defn}\cite{bm4} A function $ H:[ 0, \infty ) \times [ 0, \infty ) \rightarrow [ 0, \infty ) $ is called  $C-$class function if it  satisfies the following properties :
 \begin{enumerate}[label=(\roman*)]
 \item $H$ is continuous;
 \item $H(t,z) \leq t;$
 \item $H(t,z) = t$ implies that either  $t = 0$   or  $z = 0$, for all $z,t\in [0, \infty). $\end{enumerate}\end{defn}
 We refer\cite{bm4, bm10} for more information  on $C$-class functions.\vspace{0.2cm}\\
\textbf{Notations}: In this article, the following symbols are used to indicate:\\
 $ \mathbb{\mathbb{R^{+}}}:$  the set of all non-negative real numbers,\\
$\boldsymbol{\Xi} = \{ \xi:\mathbb{R^{+}} \rightarrow  \mathbb{R^{+}}| \xi$ is continuous, non-decreasing and $\xi^{-1}(0) = 0 \}$,\\
$\boldsymbol{\Omega} = \{ \omega : \mathbb{\mathbb{R^{+}}} \rightarrow \mathbb{R^{+}}|\omega(s_{n}) \rightarrow \infty $ implies $   s_{n} \rightarrow 0$ as $n \rightarrow \infty \}$,\\
$\boldsymbol{\Omega^{1} } = \{ \omega :  \mathbb{R^{+}} \rightarrow  \mathbb{R^{+}} | \omega (s_{n}) \rightarrow 0  \mbox{ implies }  s_{n} \rightarrow 0 \mbox{ as }  n \rightarrow \infty \mbox{ and } \omega \mbox{ is continuous} \}$,\\
$\textbf{W} = \{ 0,1,2,3,4,5,6,7,...\},$\\ 
$\textbf{P}(h,g) = \{ x : hx=gx\}$,\\
$\mathscr{C}:$ the class of all C-class functions.
\begin{defn}\cite{bm13} Let $\mathcal{O}$ be a non-empty set and $h, g$ be two self maps on $\mathcal{O}.$ A pair (h,g) is called  weakly compatible  if $h$ and $g$ commutes at the coincidence point.\end{defn}
 Chandok and Ansari\cite{bm7} introduced the concept of $TAC-$contraction in metric space in the following way:
\begin{defn}\cite{bm7} Let $(\mathcal{O},d)$ be a metric space and 
 $\gamma , \delta :\mathcal{O} \rightarrow \mathbb{R^{+}}$. An operator $\mathcal{Z} : \mathcal{O}\rightarrow \mathcal{O} $ is called  $TAC$-contractive 
map if for all $t , z \in \mathcal{O}$ with
 $\gamma(t) \delta(z) \geq 1$ implies  
 $\xi (  d( \mathcal{Z}t , \mathcal{Z}z  ) ) \leq  H( \xi ( d(t , z )) , \omega (d(t, z))) , $ 
    where $\xi \in \boldsymbol{\Xi} , \omega  \in \boldsymbol{\Omega}$ and $H \in \mathscr{C}$.
    \end{defn}
 Consequently, Babu and Dula\cite{bm5} gave the generalized TAC-contraction in $b$-metric space and proved corresponding fixed point theorems on them:
\begin{defn}\cite{bm5}  Let $(\mathcal{O},b,s)$ be a  $b$-metric space and $\gamma , \delta :
\mathcal{O} \rightarrow \mathbb{R^{+}}$ and $\mathcal{Z}$ be a self map on $\mathcal{O}$. Suppose that $\xi \in \boldsymbol{\Xi}, \omega  \in \boldsymbol{\Omega}$ and $H \in \mathscr{C}$ such that for all $t , v \in \mathcal{O}$ with $\gamma(t) \delta(v) \geq 1$   implies  $\xi ( s^{3} b( \mathcal{Z}t, \mathcal{Z}v  ) ) \leq  H( \xi ( N_{s}(t , v )) , \omega ( N_{s}(t , v))) , $ where $N_{s}(t, v) = \max \{ b(t, v) , b(t, \mathcal{Z}t), b(v , \mathcal{Z}v),\frac{b(t , \mathcal{Z}t) + b(v, \mathcal{Z}v)}{2s} \}$. Then, $\mathcal{Z}$ is called generalized $TAC-$ contraction on  $b$-metric space.
    \end{defn}
Recently, the generalized $TAC-$contraction in $b$-metric space for  four self maps was introduced by  Sastry et.al.\cite{bm15}  and gave the  fixed point theorems as follows :
\begin{defn}\cite{bm15}
Let $(\mathcal{O},b,s)$ be a  $b$-metric space and $\gamma , \delta : \mathcal{O} \rightarrow \mathbb{R^{+}}$ and $h ,\eta , Q ,\mathcal{Z}$ be  four self maps on $\mathcal{O}$. Suppose that  $\xi \in \boldsymbol{\Xi} , \omega  \in \boldsymbol{\Omega}$ and $H \in \mathscr{C}$ such that for all $t ,v \in \mathcal{O}$ with $\gamma(Qt) \delta(\mathcal{Z}v) \geq 1$ implies $ \xi ( s^{3} b( ht , \eta v ) ) \leq  H( \xi ( N_{s}(t , v )) , \omega ( N_{s}(t , v))) , $ where $N_{s}(t, v ) = \max \{ b(Qt, \mathcal{Z}v), b( ht, Qt) , $\\$b(\mathcal{Z}v , \eta v) ,
 \frac{b(Qt , \eta v) + b(ht, \mathcal{Z}v)}{2s} \}$.
 Then, the pair $(h,\eta )$ is called generalized $TAC-(Q,\mathcal{Z})$ contraction on $b$-metric space.\end{defn}
\begin{thm}\cite{bm15} Let $(\mathcal{O},b,s)$ be   a complete  $b$-metric space  and $h, \eta , Q,\mathcal{Z}$ be four self maps on $\mathcal{O}$. Assume that $\gamma ,  \delta:\mathcal{O} \rightarrow \mathbb{R^{+}} , \xi \in \boldsymbol{\Xi}, \omega \in \boldsymbol{\Omega}$ and $H \in \mathscr{C}$ such that $(h,\eta )$ is a generalized $TAC-(Q,\mathcal{Z})$ contraction with respect to $H$. Suppose that;
\begin{enumerate}[label=(\roman*)]
\item  $h\mathcal{O} \subseteq \mathcal{Z}\mathcal{O}$ and  $\eta \mathcal{O} \subseteq Q\mathcal{O}$,
\item there is $z_{0}\in Y$ such that $\gamma(Qz_{0})\geq 1$  and $\delta(\mathcal{Z}z_{0})\geq 1$,
\item if $\{ z_{k} \}_{k \in \textbf{W} } $ is a sequence in $\mathcal{O}$ with $z_{k} \rightarrow z,\gamma(z_{k})\geq  1$ and $\delta(z_{k})\geq 1,$ for  all $k\in \textbf{W} $, then $ \gamma(z) \geq 1$ and $\delta(z) \geq 1 $,
\item  $(h,\eta )$  is  cyclic $(\gamma,\delta)$-
admissible with respect to $(Q,\mathcal{Z})$,
\item one of the ranges $h\mathcal{O} , \eta \mathcal{O} ,\mathcal{Z}\mathcal{O} , Q\mathcal{O}$ is  $b$-closed.
\end{enumerate}
 Then, $\textbf{P}(h,Q)\neq \emptyset$ and $\textbf{P}(\eta ,\mathcal{Z})\neq \emptyset.$
 \end{thm} \par This paper aims to present the concept of generalized $TAC$-contraction in the setting of partial $b$-metric space, which is motivated by the concept of generalized $TAC-$contraction  in $b$-metric space given by Sastry et.al.\cite{bm15}. They have  obtained fixed point theorems for four  self maps $f,g,S,T$ in [Theorem $2.3,15$], motivated by their work, we have obtained common fixed point and  common coincidence point of four self maps for generalized $TAC-$contraction  in partial b-metric space and the viability of results are demonstrated with the help of  illustrations.
\section{Main results}
Generalized $TAC-$contraction was given by Sastri et.al.\cite{bm15} in b-metric space, therefore, we define the concept of generalized $TAC-(Q,\mathcal{Z})$ contraction in partial b-metric space in the following:
\begin{defn}Let $(\mathcal{O},\partial ,s)$ be a partial $b$-metric space, $\gamma , \delta:\mathcal{O} \rightarrow \mathbb{R^{+}}$ and $h ,\eta , Q,\mathcal{Z}$ be  four self maps on $\mathcal{O}$. Suppose that  $\xi \in \boldsymbol{\Xi} , \omega  \in \boldsymbol{\Omega^{1}}$ and $H \in  \mathscr{C}  $  with
 $\gamma(Qw) \delta(\mathcal{Z}z) \geq 1$  implies  $\xi ( s^{3} \partial( hw , \eta z  ) ) \leq H( \xi ( N_{s}(w , z )) , \omega ( N_{s}(w , z))) ,$
   where  $N_{s}(w, z ) = 
   \max \{\partial(Qw, \mathcal{Z}z) , \partial( hw, Qw) , \partial(\mathcal{Z}z , \eta z) ,
\frac{\partial(Qw , \eta z) + \partial(hw, \mathcal{Z}z)}{2s} \}$,  for all $w ,z \in \mathcal{O}$.
 Then, the pair $(h,\eta )$ is called generalized $TAC-(Q,\mathcal{Z})$ contraction in partial $b$-metric space.\end{defn}
\begin{ex} Let $\mathcal{O} = \mathbb{R^{+}} $ and define $ \partial : \mathcal{O} \times \mathcal{O} \rightarrow \mathbb{R^{+}} $ as 
$ \partial(w,z) = \max \{w, z\} $, for all  $ w , z \in \mathcal{O}.$ Clearly,  it is a partial $b$-metric space with s=1.\\
Also, define $ \gamma , \delta  : \mathcal{O} \rightarrow \mathbb{R^{+}} $ and $ h, \eta , Q , \mathcal{Z} : \mathcal{O} \rightarrow \mathcal{O} $ as \vspace{0.2cm}\\
$\gamma(z) =\begin{cases}
 \frac{\sqrt{z}}{16\sqrt{2}}, & \text{if $ z \in (0, 32 )  $ } \vspace{0.2cm}\\  2, & \text{otherwise  } 
  \end{cases},  \delta(z) = \begin{cases}
 \frac{1}{4}, & \text{if $ z \in (0, 8 )  $ } \vspace{0.2cm}\\  
 1, & \text{,otherwise } 
  \end{cases} $
  and $h(z) = z,  $ \vspace{0.2cm} \\
  $ \eta (z) =  \begin{cases}
 0, & \text{if $ z \in [ 0, 64 ) $ } \vspace{0.2cm}\\ \frac{z}{2}, & \text{otherwise} \end{cases}, \mathcal{Z}(z) = \begin{cases}
 z^{3}, & \text{if $ z \in [ 0, 64 ) $ } \vspace{0.2cm}\\  
 z^{6}, & \text{otherwise } 
  \end{cases}, ~
  Q(z)=\frac{z^{2}}{2},$\vspace{0.2cm}\\ 
   for all $z \in \mathcal{O}.   $\vspace{0.2cm}\\
Now, $\gamma(Qz)= \begin{cases}
 \frac{z}{32}, & \text{if $ z \in ( 0, 8 )  $ } \vspace{0.2cm}\\  2, & \text{otherwise } 
  \end{cases}  $ and $\delta(\mathcal{Z}z)=\begin{cases}
 \frac{1}{4}, & \text{if $ z\in ( 0, 2 )  $ }\vspace{0.2cm} \\  
1, & \text{otherwise  } 
  \end{cases} $
 \\
 Now, define $\xi , \omega:\mathbb{R^{+}} \rightarrow \mathbb{R^{+}}$ and 
 $H : \mathbb{R^{+}} \times   \mathbb{R^{+}} \rightarrow  \mathbb{R^{+}}$ as $\xi(w) = w, \omega(u)= log(u+3)$ and $ H(s, t) = \frac{1}{2}s.$  Clearly, $\xi \in \Xi ,\omega \in \Omega^{1} ,  H \in \mathscr{C}.$ \vspace{0.2cm}\\
 Now, $\gamma(Qz) \delta(\mathcal{Z}w) = \begin{cases}
     \frac{z}{128}, & \text{if $z \in (0,8)$ and $w \in (0,2)$}\vspace{0.2cm}\\
     \frac{z}{32},  & \text{if $z \in (0,8)$ and $w \in \{0\} \cup [2, \infty) $}\vspace{0.2cm}\\
     \frac{1}{2}, & \text{if $z \in \{0\} \cup [8, \infty)$ and $w \in (0,2)$} \vspace{0.2cm}\\
     2, & \text{if $z \in \{0\} \cup [8, \infty)$ and $w \in   \{0\} \cup [2, \infty) $}
 \end{cases} $ \vspace{0.2cm}\\
 Note that if $\gamma(Qz) \delta(\mathcal{Z}w) \geq 1 $, then $z \in \{0\} \cup [8, \infty)$ and $ w \in \{0\}\cup [2,\infty ). $ \\
Now, if $z \in \{0\} \cup [8, \infty)$ and $ w \in \{0\}\cup [2,64)$, then
$\xi(s^{3}\partial(hz , \eta w)) = \xi(\partial(z , 0 )) = \xi(z) = z$ and $\frac{\partial(Qz,\mathcal{Z}w)}{2} = \frac{\partial(\frac{z^{2}}{2},w^{3})}{2} \leq \frac{N_{s}(z,w)}{2} = H(  \xi(N_{s}(z,w)) ,$\\$ \omega(N_{s}(z,w)) ).$
Thus, $\xi(s^{3}\partial(hz , \eta w)) \leq  H(  \xi(N_{s}(z,w)) , \omega(N_{s}(z,w)) ),$ where $N_{s}(w, z ) =\max \{ \partial(Qw, \mathcal{Z}z) , \partial( hw, Qw) , \partial(\mathcal{Z}z , \eta z) ,
\frac{\partial(Qw , \eta z) + \partial(hw, \mathcal{Z} z)}{2s} \},$ for all $z \in \{0\} \cup [8, \infty)$ and $ w \in \{0\}\cup [2,\infty ).$
Again, if  $z \in \{0\} \cup [8, \infty)$ and $ w \in  [64,\infty ) $, then $\xi(s^{3}\partial(hz , \eta w)) = \xi( \max\{z, \frac{w}{2}\}) =\max \{z,\frac{w}{2}\}$ and $\frac{\partial(Qz,\mathcal{Z}w)}{2} = \frac{\partial(\frac{z^{2}}{2},w^{6})}{2} \leq \frac{N_{s}(z,w)}{2} = H(  \xi(N_{s}(z,w)) , \omega(N_{s}(z,w)) ).$
Thus, $\xi(s^{3}\partial(hz , \eta w)) \leq  H(  \xi(N_{s}(z,w)) , \omega(N_{s}(z,w)) ),$ for all $z \in \{0\} \cup [8, \infty)$ and $ w \in  [64,\infty ) $,
where $N_{s}(w, z ) = \max \{ \partial(Qw, \mathcal{Z}z) , \partial( hw, Qw) , \partial(\mathcal{Z}z , \eta z) ,
\frac{\partial(Qw , \eta z) + \partial(hw, \mathcal{Z}z)}{2s} \}$.
 Thus, the pair $(h,\eta )$ is  generalized $TAC-(Q,\mathcal{Z})$ contraction in partial $b$-metric space. 
\end{ex}
\begin{thm}
    Let $(\mathcal{O},\partial,s)$ be   a complete partial $b$-metric space  and $h, \eta ,Q,\mathcal{Z}$ be four self maps on $\mathcal{O}$. Suppose that there exist two mappings $\gamma ,  \delta:\mathcal{O} \rightarrow \mathbb{R^{+}}$ 
 and $\xi \in \boldsymbol{\Xi}, \omega \in \boldsymbol{\Omega^{1}}, H\in \mathscr{C}$ such that $(h,\eta )$ is a generalized $TAC-(Q,\mathcal{Z})$ contractive  mapping with respect to $H$. Assume that;
 \begin{enumerate}[label=(\roman*)]
\item $h\mathcal{O} \subseteq \mathcal{Z}\mathcal{O}$ and  $\eta \mathcal{O} \subseteq Q\mathcal{O}$.
\item there exists $v_{0}\in \mathcal{O}$ such that $\gamma(Qv_{0})\geq 1$  and $\delta (\mathcal{Z}v_{0})\geq 1$.
\item If $\{v_{m}\}_{m \in \textbf{W} }$ is a sequence in $\mathcal{O}$ such that $v_{m} \rightarrow v,\gamma(v_{m})\geq 1$ and $\delta(v_{m})\geq 1$, for  all $m\in \textbf{W}$,  then $\gamma(v)\geq 1$ and $\delta(v) \geq 1$.
\item $(h,\eta )$  is  cyclic $(\gamma,\delta)$-admissible with respect to $(Q,\mathcal{Z})$.
\item one of the ranges $h\mathcal{O} , \eta \mathcal{O} , \mathcal{Z}\mathcal{O} , Q\mathcal{O}$ is partially $b$-closed.
\end{enumerate}
Then, $\textbf{P}(h,Q) \neq \emptyset$ and $\textbf{P}(\eta ,\mathcal{Z})\neq \emptyset$ 
\end{thm}
\begin{proof}
By $(ii)$, $v_{0}$ exists and by $(i)$,  define $\{d_{m}\}_{m \in \textbf{W}}$  as  \begin{align} 
d_{2m}= hv_{2m}= \mathcal{Z}v_{2m+1}\mbox{ and }    d_{2m+1} =  Qv_{2m+2} =  \eta v_{2m+1}.\hspace{1.9cm}\end{align} 
 Claim $1$. $\{d_{m}\}_{m \in \textbf{W} }$ is a Cauchy sequence in $\mathcal{O}$.\\
Since $\gamma(Qv_{0})\ge1$ and $(h,\eta)$ is cyclic $(\gamma,\delta)$-admissible with respect to $(Q,\mathcal{Z})$, $\delta(hv_{0})\ge1$. So, $\delta(\mathcal{Z}v_{1})\ge1$ by $(2.1)$. Thus, $\gamma(\eta v_{1}) \ge1$ by  $(iv)$. So, $\gamma(Qv_{2})\geq 1$ by $(2.1).$ On recurring the process, we get\begin{align} \gamma (Qv_{2m}) \geq 1 \mbox{ and } \delta(\mathcal{Z}v_{2m+1})\geq1.\hspace{5.4cm}\end{align}
If $d_{2m} = d_{2m+1},$ then \begin{align*}  
N_{s}(v_{2m+2},v_{2m+1}) = \max \{\partial(Qv_{2m+2},\mathcal{Z}v_{2m+1}),\partial( hv_{2m+2}, Qv_{2m+2}),\hspace{1cm} \end{align*} \begin{align*}
&\hspace{5cm} \partial(\mathcal{Z}v_{2m+1}, \eta v_{2m+1}), \\ &\hspace{4cm}\frac{\partial(Qv_{2m+2} , \eta v_{2m+1}) + \partial(hv_{2m+2}, \mathcal{Z}v_{2m+1})}{2s}\} \\&\hspace{2cm}= \max\{\partial(d_{2m+1} , d_{2m}) ,\partial(d_{2m+2}, d_{2m+1}) , \partial(d_{2m} , d_{2m+1}),\\   &\hspace{4cm}\frac{\partial(d_{2m+1} , d_{2m+1}) + \partial(d_{2m+2}, d_{2m})}{2s} \} \\  &\hspace{2cm}\leq \max\{\partial(d_{2m+1} , d_{2m}) ,\partial(d_{2m+2}, d_{2m+1}), \\ &\hspace{4cm}\frac{\partial(d_{2m+1} , d_{2m+1}) + \partial(d_{2m+2}, d_{2m})}{2}\}\\ &\hspace{2cm}=\max\{\partial(d_{2m+1} , d_{2m}) ,\partial(d_{2m+2}, d_{2m+1}), \\ &\hspace{4cm}\frac{\partial(d_{2m+1} , d_{2m}) + \partial(d_{2m+2}, d_{2m+1})}{2}\}\\  &\hspace{2cm}=\max\{\partial(d_{2m+1} , d_{2m}) ,\partial(d_{2m+2}, d_{2m+1})\}\\  &\hspace{2cm}= \max\{\partial(d_{2m+1} , d_{2m+1}) ,\partial(d_{2m+2}, d_{2m+1})\}\\  &\hspace{2cm}\leq  \partial(d_{2m+2},d_{2m+1}).\end{align*}
 Hence, $N_{s}(v_{2m+2},v_{2m+1}) = \partial(d_{2m+2} , d_{2m+1}).$\\
Since $\gamma(Qv_{2m+2}) \delta(\mathcal{Z}v_{2m+1})\geq 1$ by inequality $(2.2)$, therefore, 
\begin{align*}\xi(\partial(d_{2m+2} , d_{2m+1})) &\leq  \xi(s^{3}\partial(hv_{2m+2}, \eta v_{2m+1})) \\ &\leq H(\xi(N_{s}(v_{2m+2} ,v_{2m+1}) ) , \omega(N_{s}(v_{2m+2} ,v_{2m+1}) )  ) \hspace{1.5cm}\\ 
 &= H(\xi(\partial(d_{2m+2} ,d_{2m+1}) ) , \omega(\partial(d_{2m+2} ,d_{2m+1}) ) ). \end{align*}
Also, $H(\xi(\partial(d_{2m+2} ,d_{2m+1}) ) , \omega(\partial(d_{2m+2} ,d_{2m+1}) )  ) \leq \xi(\partial(d_{2m+2} ,d_{2m+1}) )$\\ 
so, $H(\xi(\partial(d_{2m+2} ,d_{2m+1}) ) , \omega(\partial(d_{2m+2} ,d_{2m+1}) )  ) = \xi(\partial(d_{2m+2} ,d_{2m+1}) ).$\\
Thus, by the definition of $H$, we have $\xi(\partial(d_{2m+2} ,d_{2m+1}) ) = 0$  or $\omega(\partial(d_{2m+2} ,$\\$d_{2m+1}) ) = 0$. Also, by the definition of $\xi$ and  $\omega$, we get  
$\partial(d_{2m+2} ,d_{2m+1}) = 0$. Hence, $d_{2m+1} = d_{2m+2}.$ Thus, $ d_{2m} = d_{2m+1} = d_{2m+2}.$ 
 On iteration, we can  prove  that $d_{2m} = d_{2m+1}$, for all $m \in \textbf{W}$.
 Hence, $\{d_m\}_{m \in \textbf{W}}$ is a constant sequence.
 Therefore, $  \textbf{P}(h,Q)\neq \emptyset$ and $\textbf{P}(\eta ,\mathcal{Z})\neq \emptyset$.\\
 So, without loss of generality  assume that  $z_{2m} \neq z_{2m+1}$ for all $m \in \textbf{W}.$
\begin{align}
&\mbox{Claim 1(a).}\lim_{m\to\infty} \partial(z_{m}, z_{m+1}) = 0.\\ 
 &\mbox{Now, } N_{s}(v_{2m},v_{2m+1}) = \max\{\partial(Qv_{2m},\mathcal{Z}v_{2m+1}),\partial( hv_{2m}, Qv_{2m}) ,\notag \\ &\hspace{8cm}\partial(\mathcal{Z}v_{2m+1} , \eta v_{2m+1} ) 
\notag\\&\hspace{5cm}\frac{\partial(Qv_{2m} , \eta v_{2m+1}) + \partial(hv_{2m}, \mathcal{Z}v_{2m+1})}{2s}\} \notag\\ &\hspace{3.2cm}= \max\{\partial(d_{2m-1} , d_{2m}) ,\partial(d_{2m}, d_{2m-1}) , \partial(d_{2m} , d_{2m+1}),\notag \\&\hspace{5cm}
\frac{\partial(d_{2m-1} , d_{2m+1}) + \partial(d_{2m}, d_{2m})}{2s} \} \notag\\&\hspace{3.2cm}\leq \max\{\partial(d_{2m} , d_{2m-1}) ,\partial(d_{2m}, d_{2m+1}),\notag\\&
\hspace{5cm}\frac{s(\partial(d_{2m-1} , d_{2m}) + \partial(d_{2m}, d_{2m+1}))}{2s}\}\notag\end{align}
\begin{align*}&
\hspace{1.5cm}=\max\{\partial(d_{2m}, d_{2m-1}) ,\partial(d_{2m}, d_{2m+1})\\ &
\hspace{3.5cm} \frac{\partial(d_{2m-1} , d_{2m}) + \partial(d_{2m}, d_{2m+1})}{2}\}\\
&\hspace{1.5cm}=\max\{\partial(d_{2m-1} , d_{2m}) ,\partial(d_{2m}, d_{2m+1})\}.
\end{align*}
If $ \partial(d_{2m} , d_{2m+1}) > \partial(d_{2m} , d_{2m-1} )$,  then $N_{s}(v_{2m},v_{2m+1}) = \partial(d_{2m} , d_{2m+1}).$ \\
By inequality $(2.2)$, we have $ \gamma(Qv_{2m}) \delta(\mathcal{Z}v_{2m+1})\geq 1$ implies  
\begin{align}
  \xi(\partial(d_{2m} , d_{2m+1}) &\leq \xi(s^{3}\partial(hv_{2m} , \eta v_{2m+1})\hspace{2cm}\notag \\
  &\leq H(\xi(N_{s}(v_{2m} ,v_{2m+1}) ) , \omega(N_{s}(v_{2m} ,v_{2m+1}) )  ) \\ &=H(\xi(\partial(d_{2m} ,d_{2m+1}) ) , \omega(\partial(d_{2m} ,d_{2m+1}) )  )\notag\hspace{3.2cm} \end{align}
Also,  $H(\xi(\partial(d_{2m} ,d_{2m+1}) ) , \omega(\partial(d_{2m} ,d_{2m+1}) )  )  \leq \xi(\partial(d_{2m} ,d_{2m+1}) ),$\\
so, $H(\xi(\partial(d_{2m} ,d_{2m+1}) ) , \omega(d_{2m} ,d_{2m+1}) )  ) = \xi(\partial(d_{2m} ,d_{2m+1}) ).$\\
By the definition of H, we get $ \xi(\partial(d_{2m} ,d_{2m+1}) ) = 0$         or     $\omega(\partial(d_{2m} ,d_{2m+1}) ) = 0 $,\\
 thus, by the definition of $\xi$ and  $\omega$, we get $\partial(d_{2m} ,d_{2m+1}) = 0.$ So, $d_{2m} = d_{2m+1}$, which is absurd as $d_{2m} \neq d_{2m+1}$  for all $m \in \textbf{W}$. 
\begin{equation}
\mbox{So, }\partial(d_{2m} , d_{2m+1}) \leq \partial(d_{2m} , d_{2m-1})\mbox{ and }N_{s}(v_{2m},v_{2m+1}) = \partial(d_{2m} , d_{2m-1}).\end{equation}
Hence,  $\{ \partial(d_{2m} , d_{2m+1}) \}_{m \in \textbf{W}} $ is a non-increasing sequence of  real numbers.
So, it has a limit $ v \geq 0$ (say).
Also, by inequality $(2.4)$ and $(2.5)$,  we get \\ $\xi(\partial(d_{2m} , d_{2m+1}) \leq H(\xi(\partial(d_{2m} ,d_{2m-1}), \omega(\partial(d_{2m} ,d_{2m-1})  ) \leq   \xi(\partial(d_{2m} ,d_{2m-1})$ \\ 
Taking limit  and using continuity of $ H , \xi$   and $\omega$, we get 
$ \xi(v) \leq G( \xi(v) , \lim_{m\to\infty}$\\$ \omega(\partial(d_{2m} ,d_{2m-1}))) \leq \xi(v).$ 
So, $H( \xi(v) , \lim_{m\to\infty} \omega(\partial(d_{2m} , d_{2m-1}))) =\xi(v).$ Hence, $\xi(v) = 0 $  or 
$\lim_{m\to\infty} \omega(\partial(d_{2m} ,d_{2m-1})) =0  $. 
By the property of  $\xi$ and $\omega$, we get $v = 0.$ 
Thus, $\lim_{m\to\infty} \partial(d_{m}, d_{m+1}) = 0$.\begin{equation}
\mbox{Claim 1(b). } \{d_{2m}\}_{m \in \textbf{W}} \mbox{ is a Cauchy sequence.}\hspace{3.6cm}\end{equation}
Suppose $\{d_{2m}\}_{m \in \textbf{W}}$ is not  Cauchy. So,  there exists $\epsilon > 0 ,\{d_{2n(l)}\} $ and $ 
\{d_{2m(l)}\}$ subsequences of $ \{d_{2m}\} $
where  $ n(l) $ 
is the smallest  integer such that
\begin{align}
 n(l) > m(l) \geq l,~ \partial( d_{2n(l)} , d_{2m(l)} )\geq\epsilon  \mbox{ and }
\partial( d_{2n(l)-2} , d_{2m(l)} )<\epsilon.\end{align}
Using inequalities $(2.3)$ and $(2.7)$, we get 
\begin{align}&\epsilon \leq \partial(d_{2n(l)},d_{2m(l)})\leq s(\partial(d_{2n(l)}+d_{2n(l)-2})+\partial(d_{2n(l)-2},d_{2m(l)}))\hspace{1.5cm}\notag\\
  &\hspace{0.15cm}\leq s^{2}(\partial(d_{2n(l)}+d_{2n(l)-1})+\partial(d_{2n(l)-1},d_{2n(l)-2})) +s \epsilon, \notag \\*& 
\mbox{so, } \epsilon \leq \limsup_{l\rightarrow\infty} \partial(d_{2n(l)},d_{2m(l)}) \leq s \epsilon. \end{align}
Again, $\partial(d_{2n(l)},d_{2m(l)+1})\leq s(\partial(d_{2n(l)}+d_{2m(l)})+\partial(d_{2m(l)},d_{2m(l)+1})).$\\
Using inequalities $(2.3)$ and  $(2.8),$ we get \begin{align} &\limsup_{l\rightarrow\infty} \partial(d_{2n(l)},d_{2m(l)+1}) \leq s^{2} \epsilon.\\& 
\mbox{Also, } \epsilon \leq  \partial(d_{2n(l)},d_{2m(l)}) \leq  s(\partial(d_{2n(l)}, d_{2m(l)+1})+\partial(d_{2m(l)+1},d_{2m(l)})).\hspace{2cm}\notag\end{align}
Using inequalities  $(2.3)$ and  $(2.9),$  we get  
 \begin{equation}\frac{ \epsilon }{ s } \leq  \limsup_{l\rightarrow\infty} \partial(d_{2n(l)},d_{2m(l)+1}) \leq s^{2} \epsilon. \hspace{4.6cm}\end{equation} 
Now, $\partial(d_{2n(l)-1},d_{2m(l)}) \leq  s(\partial(d_{2n(l)-1}, d_{2n(l)})+\partial(d_{2n(l)},d_{2m(l)})).$ \\
Using inequalities $(2.3)$ and $(2.8)$, we get \begin{equation}
\limsup_{l\rightarrow\infty} \partial(d_{2n(l)-1},d_{2m(l)}) \leq s^{2} \epsilon. \hspace{5.3cm} \end{equation}
 Also, $\epsilon \leq \partial(d_{2n(l)},d_{2m(l)}) \leq  s(\partial(d_{2m(l)}, d_{2n(l)-1})+\partial(d_{2n(l)-1},d_{2n(l)}))$. \\ 
 Using inequalities $(2.3)$ and $(2.11)$, we get \begin{align}&\frac{\epsilon}{s} \leq  \limsup_{l\rightarrow\infty} \partial(d_{2n(l)-1},d_{2m(l)}) \leq s^{2} \epsilon.\\&
\mbox{Now, } \partial(d_{2m(l)+1},d_{2n(l)-1}) \leq  s(\partial(d_{2m(l)+1}, d_{2m(l)})+\partial(d_{2m(l)},d_{2n(l)-1}))\hspace{1cm}\notag\\
& \hspace{4cm} \leq s(\partial(d_{2m(l)+1}, d_{2m(l)}))+s^{2}(\partial(d_{2m(l)},d_{2n(l)})) +\notag\\& \hspace{5cm}s^{2}(\partial(d_{2n(l)} , d_{2n(l)-1} ) ).\notag\end{align}
Using inequalities $(2.3)$ and $(2.7)$, we get  \begin{align*} &\limsup_{l\rightarrow\infty} \partial(d_{2n(l)-1},d_{2m(l)+1}) \leq s^{3} \epsilon.\hspace{6.5cm}\\
  &\mbox{Since } \epsilon \leq \partial(d_{2n(l)},d_{2m(l)}) \leq  s(\partial(d_{2n(l)}, d_{2m(l)+1})+\partial(d_{2m(l)+1},d_{2m(l)}))\\&\hspace{1cm}\leq s(\partial(d_{2n(l)}, d_{2n(l)-1})) + s^{2}(\partial(d_{2n(l)-1},d_{2m(l)+1}) + \partial(d_{2m(l)+1},d_{2m(l)})).\end{align*} 
Therefore, using inequalities $(2.3)$ and $(2.13)$, we get \begin{align}\frac{\epsilon}{s^{2}} \leq  \limsup_{l\rightarrow\infty} \partial(d_{2n(l)-1},d_{2m(l)+1}) \leq s^{3} \epsilon. \hspace{4.1cm} \end{align}
\begin{align*}
\mbox{Now, } N_{s}(v_{2n(l)},v_{2m(l)+1})  &= \max \{ \partial(Qv_{2n(l)},\mathcal{Z}v_{2m(l)+1}) , \partial( hv_{2n(l)}, Qv_{2n(l)}),\\ &\hspace{1cm}\partial(\mathcal{Z}v_{2m(l)+1} , \eta v_{2m(l) +1} ),\\
&\hspace{1cm}\frac{\partial(Qv_{2n(l)} , \eta v_{2m(l)+1}) + \partial(hv_{2n(l)}, \mathcal{Z}v_{2m(l)+1})}{2s} \} \\
 &=\max \{\partial(d_{2n(l)-1} , d_{2m(l)}) ,\partial(d_{2n(l)}, d_{2n(l)-1}),\\ &\hspace{1.5cm}\partial(d_{2m(l)} , d_{2m(l)+1}), \\
&\hspace{1cm}\frac{\partial(d_{2n(l)-1} , d_{2m(l)+1}) + \partial(d_{2n(l)}, d_{2m(l)}}{2s}\}.\end{align*}
So, $\partial(d_{2n(l)-1} , d_{2m(l)}) \leq N_{s}(v_{2n(l)},v_{2m(l)+1}) \leq \max\{\partial(d_{2n(l)-1} , d_{2m(l)}) , \\ 
  \partial(d_{2n(l)}, d_{2n(l)-1}) , \partial(d_{2m(l)} , d_{2m(l)+1}), 
\frac{\partial(d_{2n(l)-1} , d_{2m(l)+1})  + \partial(d_{2n(l)}, d_{2m(l)})}{2s} \}.$\\
Using inequalities $ (2.3) , (2.8) , (2.12) ,$ and $(2.14) ,$ we get \\
$ \frac{\epsilon}{s} \leq \limsup_{l\rightarrow\infty}  \partial(d_{2n(l)-1},d_{2m(l)}) \leq \limsup_{l\rightarrow\infty} N_{s}(v_{2n(l)},v_{2m(l)+1}) \leq $\\$\max \{ s^2 \epsilon , \frac{1}{2s}(s^{3} \epsilon + s^ {2} \epsilon )\} \leq \max\{ s^2 \epsilon , \frac{1}{2s}(s^{3} \epsilon + s^ {3} \epsilon )\}$.\\
So, $\frac{\epsilon}{s} \leq \limsup_{l\rightarrow\infty} N_{s}(v_{2n(l)},v_{2m(l)+1})\leq s^{2} \epsilon.$\\
 Therefore, there exists a subsequence $\{N_{s}(v_{2n(l_{t})},v_{2m(l_{t})+1})\}$ such that 
\begin{equation}\frac{\epsilon}{s} \leq  \lim_{t\rightarrow\infty} N_{s}(\partial(v_{2n(l_{t})},v_{2m(l_{t})+1})) \leq  s^{2} \epsilon.\hspace{4.15cm}\end{equation} Since $\gamma(Qv_{2n(l_{t})}) \delta(\mathcal{Z}v_{2m(l_{t})+1})\geq 1$ by inequality $(2.2)$, therefore,   \begin{align} \xi(s^{3} \partial( hv_{2n(l_{t})} , \eta v_{2m(l_{t})+1})) \leq H( \xi ( N_{s}&(v_{2n(l_{t})} , v_{2m(l_{t})+1} )) , \notag \\ &\hspace{0.5cm}\omega ( N_{s}(v_{2n(l_{t})} , v_{2m(l_{t})+1} ))).      \end{align}
Also, by the contiuity of  $\xi$,  inequalities $(2.10)$ and $(2.15)$,  we have 
\begin{align}
\xi (\lim_{t\rightarrow\infty} N_{s}(v_{2n(l_{t})} , v_{2m(l_{t})+1} ))  &\leq \xi ( s^{3} \frac {\epsilon}{s} )\notag \\ &\leq \xi (s^{3} \lim_{t\rightarrow\infty} \partial(d_{2n(l_{t})} , d_{2m(l_{t})+1} )).\hspace{0.7cm}\end{align}
Also, by the continuity of  $ H , \xi ,$ and $ \omega$ in $(2.16)$,  we have
\begin{align} \xi(s^{3} \lim_{t\rightarrow\infty} \partial( hv_{2n(l_{t})} , \eta v_{2m(l_{t})+1})) \leq H( \xi (\lim_{t\rightarrow\infty} N_{s}(v_{2n(l_{t})} , v_{2m(l_{t})+1} )) ,\notag \\ \lim_{t\rightarrow\infty} \omega ( N_{s}(v_{2n(l_{t})} , v_{2m(l_{t})+1}))) \notag \\ \leq \xi(\lim_{t\rightarrow\infty}   N_{s}(v_{2n(l_{t})} , v_{2m(l_{t})+1} ) ).\hspace{0.3cm} \end{align}
Also, by  inequalities $(2.17)$ and $(2.18)$, we have 
$H(\xi(\lim_{t\rightarrow\infty} N_{s}(v_{2n(l_{t})} ,$\\$v_{2m(l_{t})+1} )) , \lim_{t\rightarrow\infty}  \omega ( N_{s}(v_{2n(l_{t})} ,v_{2m(l_{t})+1})))  
 = \xi ( \lim_{t\rightarrow\infty}   N_{s}(v_{2n(l_{t})} ,$\\$ v_{2m(l_{t})+1} ) ) ) .$
 By the definition of $ H  $, we have 
$\xi  ( \lim_{t\rightarrow\infty}   N_{s}(v_{2n(l_{t})} , v_{2m(l_{t})+1} ) ) ) $\\$= 0 $ or $\lim_{t\rightarrow\infty} \omega ( N_{s}(v_{2n(l_{t})} , v_{2m(l_{t})+1})))=0.$ 
So, by the definition of $\xi$ and $\omega$, we have $ \lim_{t\rightarrow\infty}  N_{s}(v_{2n(l_{t})} , v_{2m(l_{t})+1})= 0 $.
By inequality  $(2.15)$, we have  $\frac{\epsilon}{s} = 0 $ implies  $\epsilon =0 $, which is absurd as $ \epsilon > 0 $. 
Thus, $\{d_{2m}\}_{m\in \textbf{W}}$ is a Cauchy sequence.\\ 
For claim $1. \mbox{ For all } k , l \geq 0 $, we will only examine the following  cases: \\ 
Case $(i).$ For $ k= 2n+1 ,  l = 2m ,$ for all $n,m \geq 0 $, we have \\
$\partial(d_{k} ,d_{l} ) = \partial(d_{2n+1} , d_{2m} ) \leq s(\partial(d_{2n+1}, d_{2m}) + \partial(d_{2n} , d_{2m} ) ). $ By $(2.3)$ and $(2.6)$, we have $\lim_{k, l \rightarrow\infty}\partial(d_{k} ,d_{l}) = 0 $.\\
Case $(ii).\mbox { For }  k= 2n+1 , l= 2m+1 , \mbox{ for all } n, m \geq 0$, then\\
$\partial(d_{k} , d_{l}) = \partial(d_{2n+1}, d_{2m+1}) \leq s(\partial(d_{2n+1}, d_{2n})) + s^{2}(\partial(d_{2n}, d_{2m})) + $\\$s^{2}(\partial(d_{2m}, d_{2m+1})). $
By $(2.3)$ and  $(2.6)$, we have
$\lim_{k, l \rightarrow\infty}\partial(d_{k} ,d_{l}) = 0 $.\\
Thus, $\{d_{m}\}_{m\in \textbf{W}}$ is a Cauchy sequence.
Since $(\mathcal{O}, \partial , s )$ is complete, 
there exist $d \in \mathcal{O}$ such that 
$\lim_{m \rightarrow\infty} d_{m} = d. $
\begin{equation}\mbox{So,} \lim_{k \rightarrow\infty} hv_{2k} = \lim_{k \rightarrow\infty}\mathcal{Z}v_{2k+1}= \lim_{k \rightarrow\infty} Qv_{2k+2} = \lim_{k \rightarrow\infty}\eta v_{2k+1} =d.\hspace{0.6cm}\end{equation} 
Now, we examine the following cases:\\
Case (I). Assume $Q\mathcal{O}$ is closed.
So, $ d \in Q\mathcal{O}$, there exist $w \in\mathcal{O} $ such that $ Qw = d . $ \\
Claim $2$. $hw= d.$
\begin{equation}  \mbox{Now, }\gamma (Qw)\geq 1  \mbox{ and }  \delta(\mathcal{Z}v_{2m+1}) \geq 1 \mbox{ by }(iii) \mbox{ and  inequality } (2.2).\hspace{0.2cm}\end{equation}
Also, $\partial(hw , \eta v_{2m+1}) \leq s (\partial(hw , d) + \partial(d, \eta v_{2m+1})).$
Using equations  $(2.3)$ and \begin{align} (2.19),  \mbox{ we have } \limsup_{m \rightarrow\infty} \partial(hw , \eta v_{2m+1}) \leq s (\partial(hw , d) ).\hspace{2.2cm} \end{align}
Again,  $\partial(hw , d)  \leq s (\partial(hw , \eta v_{2m+1}) + \partial(\eta v_{2m+1}, d )).$ 
Using equations $(2.3)$  \begin{equation}\mbox{and } (2.19), \mbox{ we have } \frac{1}{s} \partial(hw , d) \leq \limsup_{m \rightarrow\infty} \partial(hw , \eta v_{2m+1}).\hspace{1.7cm}\end{equation}
By equations $(2.21)$ and $(2.22)$, we get a subsequence $\{\partial(hw , \eta v_{2m_{k}+1})\}$ with   
\begin{equation}\frac{1}{s}\partial(hw , d)\leq\lim_{k \rightarrow\infty} \partial(hw , \eta v_{2m_{k}+1}) \leq s (\partial(hw , d)). \hspace{2.9cm}\end{equation}
\begin{align*}\mbox {Now, } N_{s}(w,v_{2m_{k}+1}) &= \max\{\partial(Qw,\mathcal{Z}v_{2m_{k}+1}),\partial( hw, Qw) , \partial(\mathcal{Z}v_{2m+1} , \eta v_{2m+1} ),\\ 
&\hspace{1.2cm}\frac{\partial(Qw , \eta v_{2m_{k}+1}) + \partial(hw, \mathcal{Z}v_{2m_{k}+1})}{2s}\}\\
 &\leq \max\{\partial(Qw , \mathcal{Z}v_{2m_{k}+1}) ,\partial(hw, Qw) , \partial( \mathcal{Z}v_{2m+1} , \eta v_{2m_{k}+1}),
\\ &\hspace{1cm}\frac{\partial(Qw , \eta v_{2m_{k}+1}) + s(\partial(hw, Qw) + \partial(Qw, \mathcal{Z}v_{2m_{k}+1})) }{2s}\} \\ &\leq \max\{\partial(Qw , \mathcal{Z}v_{2m_{k}+1}) ,\partial(hw, Qw) , \partial(\mathcal{Z}v_{2m_{k}+1} , \eta v_{2m_{k}+1}),\\ 
&\hspace{1.2cm}\frac{\partial(Qw , \eta v_{2m_{k}+1})] + \partial(hw, Qw) + \partial(Qw, \mathcal{Z}v_{2m_{k}+1})}{2s}\}.\end{align*}
Using equations $(2.3)$ and $(2.19)$, we have $\lim_{k \rightarrow\infty} N_{s}(w, v_{2m_{k}+1}) = \partial( hw , Qw ). $\\
Since $\gamma (Qw) \delta(\mathcal{Z}v_{2m_{k}+1}) \geq 1$ from equation $(2.20)$, therefore, \begin{align}
 &\xi ( s^{3} \partial( hw , \eta v_{2m+1}  ) )\leq  H( \xi ( N_{s}(w , v_{2m_{k}+1} )) , \omega ( N_{s}(w , v_{2m_{k}+1}))).\hspace{0.7cm}
 \\ & \mbox{As, }\xi (\partial(hw , Qw)) \leq \xi( s^{3} \frac{1}{s} \partial( hw , Qw)) \leq  \xi ( s^{3} \lim_{k \rightarrow\infty}\partial( hw , \eta v_{2m_{k}+1} ) ).\hspace{0.05cm}\end{align}
 Using the continuity of $ H, \xi ,\omega$, 
 and equations $(2.24)$ and $(2.25)$,  we get\\
 $\xi (\partial(hw , Qw)) \leq  H( \xi (\lim_{k \rightarrow\infty} N_{s}(w , v_{2m_{k}+1} )) , \lim_{k \rightarrow\infty}\omega ( N_{s}(w , v_{2m_{k}+1}))) \leq \xi ( \partial( hw , Qw ) ). $
So, $H( \xi (\partial( hw , Qw )) , \lim_{k \rightarrow\infty}\omega ( N_{s}(w , v_{2m_{k}+1}))) = \xi ( \partial( hw , Qw ) ).$ \\
 Using the definition of $ H, \xi$ and $ \omega$,  we get $\lim_{k \rightarrow\infty} \omega ( N_{s}(w , v_{2m_{k}+1})) = 0$  or
 $\xi ( \partial( hw , Qw ) ) = 0$   implies  $\partial(hw, Qw) = 0$ or $\lim_{k \rightarrow\infty}  N_{s}(w , v_{2m_{k}+1})= 0$.
Hence, $\partial(hw , Qw)  = 0.$ 
 So, $hw = Qw = d. $ 
As $d =hw \in h\mathcal{O} \subseteq \mathcal{Z}\mathcal{O},$   we get $x \in \mathcal{O} $ such that $ \mathcal{Z}x = d. $ \\
Claim $3$. $\eta x =\mathcal{Z}x. $
\begin{equation}
\mbox{Now, }\gamma (Qv_{2m}) \geq 1 \mbox{ and } \delta (\mathcal{Z}x) \geq 1 , \mbox{ by } (iii) \mbox{  and equation } (2.2).\hspace{0.7cm}\end{equation}
Also, $\partial( hv_{2m} , \eta x  ) \leq s (\partial(hv_{2m} , \mathcal{Z}x) + \partial(\mathcal{Z}x , \eta x )).$
Using equations $(2.3)$ and \begin{equation}(2.19),  \mbox{ we have }\limsup_{m \rightarrow\infty} \partial(hv_{2m} , \eta x) \leq s(\partial(\mathcal{Z}x , \eta x ) ).\hspace{2.4cm} \end{equation}
Again, $\partial(\mathcal{Z}x, \eta x)  \leq s (\partial(\mathcal{Z}x , hv_{2m}) + \partial(hv_{2m}, \eta x )).$ 
Using $(2.3)$  and $(2.19)$, we \begin{equation}\mbox{get }\frac{1}{s} \partial(\mathcal{Z}x , \eta x) \leq \limsup_{m \rightarrow\infty} \partial( hv_{2m} , \eta x). \hspace{4.5cm}\end{equation}
From  equations $(2.28)$ and  $(2.29)$, we get a subsequence $\{\partial(hv_{2m_{k}} , \eta x)\}$ with  
\begin{align}&\frac{1}{s} \partial(\mathcal{Z}x , \eta x) \leq \lim_{k \rightarrow\infty} \partial(hv_{2m_{k}} , \eta x) \leq s (\partial(\mathcal{Z}x , \eta x) ).\\
&\mbox{Now, } N_{s}(v_{2m_{k}}, x ) = \max\{\partial(Qv_{2m_{k}},\mathcal{Z}x),\partial( Qv_{2m_{k}}, hv_{2m_{k}}) , \partial( \mathcal{Z}x , \eta x),\hspace{3cm}\notag\\&\hspace{4.4cm}\frac{\partial(Qv_{2m_{k}} , \eta x) + \partial(hv_{2m_{k}}, \mathcal{Z}x )}{2s}\}\notag\\ 
&\hspace{2.7cm}\leq \max\{\partial(Qv_{2m_{k}} , \mathcal{Z}x) ,\partial(Qv_{2m_{k}}, hv_{2m_{k}}) , \partial( \mathcal{Z}x , \eta x),\notag\end{align}
\begin{align}
&\hspace{3cm}\frac{\partial(Qv_{2m_{k}} , \mathcal{Z}x) + s(\partial(\mathcal{Z}x, \eta x) + \partial(hv_{2m_{k}}, \mathcal{Z}x))}{2s}\}\notag \\ &\hspace{2cm}\leq \max\{\partial(Qv_{2m_{k}} , \mathcal{Z}x) ,\partial(Qv_{2m_{k}}, hv_{2m_{k}}) , \partial(\mathcal{Z}x , \eta x), 
\notag\\&\hspace{3cm}\partial(Qv_{2m_{k}}, \mathcal{Z}x) +\partial(\mathcal{Z}x, \eta x) +\frac{\partial(hv_{2m_{k}} , \mathcal{Z}x)}{2} \}.\notag\\
\mbox{Using }& (2.3) \mbox{ and } (2.19),\mbox{ we have }
\lim_{k \rightarrow\infty} N_{s}( v_{2m_{k}} , x ) = \partial( \mathcal{Z}x , \eta x ).\hspace{0.7cm}\end{align}
Since
 $\gamma (Qv_{2m_{k}}) \delta(\mathcal{Z}x) \geq 1 $ from equation $(2.27)$, therefore, \begin{equation}
 \xi ( s^{3} \partial( hv_{2m_{k}} , \eta x  ) )\leq  H( \xi ( N_{s}(v_{2m_{k}} , x )) , \omega ( N_{s}(v_{2m_{k}} , x))).\hspace{1.8cm}\end{equation}
 Also, by inequality $(2.29)$, we have 
\begin{equation} \xi (\partial(\mathcal{Z}x , \eta x)) \leq \xi( s^{3} \frac{1}{s} \partial(\mathcal{Z}x , \eta x )) \leq  \xi ( s^{3} \lim_{k \rightarrow\infty}\partial( hv_{2m_{k}}, \eta x ) )\hspace{1.4cm}\end{equation}
 Using the continuity of $ H, \xi , \omega ,$ equations $ (2.31)$ and $ (2.32)$,  we get\\
 $\xi (\partial(\mathcal{Z}x , \eta x)) \leq  H( \xi (\lim_{k \rightarrow\infty} N_{s}(v_{2m_{k}} , x ) ) , \lim_{k \rightarrow\infty} \omega ( N_{s}(v_{2m_{k}} , x )))\leq $\\ $\xi ( \partial( \mathcal{Z}x ,\eta x ) ).$ So, $ H( \xi (\partial( \mathcal{Z}x , \eta x)) , \lim_{k \rightarrow\infty}\omega ( N_{s}(v_{2m_{k}} , x))) = \xi ( \partial(  \mathcal{Z}x , \eta x ) ). $
 Using the definition of $ H , \xi$ and $\omega$,  we get 
 $\xi ( \partial(  \mathcal{Z}x , \eta x ) ) = 0$ or  $\lim_{k \rightarrow\infty}$\\$\omega ( N_{s}(v_{2m_{k}} , x)) = 0.$
 So, $\partial( \mathcal{Z}x , \eta x) = 0$ or $\lim_{k \rightarrow\infty}  N_{s}(v_{2m_{k}} , x) = 0.$
 Hence, $\partial( \mathcal{Z}x , \eta x) = 0.$
 Thus,  $\mathcal{Z}x = \eta x = d. $
Therefore, $Qw = hw = \mathcal{Z}x = \eta x = d$.
Hence,  $ \textbf{P}(h,Q)\neq \emptyset$ and $\textbf{P}(\eta ,\mathcal{Z})\neq \emptyset$. \\
Case (II). Assume $\eta \mathcal{O}$ is closed.
So, $d \in \eta \mathcal{O} $. As $\eta \mathcal{O} \subseteq Q\mathcal{O}$,  $ d \in Q \mathcal{O}$. 
So, we get $  w \in \mathcal{O}$ such that $d= Qw$.
Hence, the same proof follows as in case (I).\\
 Case (III). Assume $\mathcal{Z}\mathcal{O}$ is closed. 
 So, $d \in \mathcal{Z}\mathcal{O} .$ Thus, we get $x \in \mathcal{O} $ such that $\mathcal{Z} x = d$. We can show $\eta x = d $ as in case (I), then $d\in  Q\mathcal{O}$. 
 Hence, the same proof  follows as in case (I).\\
 Case (IV). Assume $h\mathcal{O}$ is closed.
 So, $ d \in h\mathcal{O}$. As, $h \mathcal{O}\subseteq \mathcal{Z}\mathcal{O}$, we get $x \in\mathcal{O} $ such that $\mathcal{Z}x = d$. 
Hence, the same proof follows as in case (III).
\end{proof}
If we consider the weak compatibility of pairs $(h,Q)$ and $(\eta ,\mathcal{Z})$, then  they have coincidence points which is shown in the following:
\begin{thm} If the following conditions are added to Theorem $2.2$:
\begin{enumerate}[label=(\roman*)]
\item $( h, Q )$ and $( \eta , \mathcal{Z})$ are weakly compatible,
\item If the pairs $(h, Q)$ and $(\eta, \mathcal{Z})$  has coincidence points $w$ and $x$ respectively, then $\gamma (Qw) \geq 1$ and $\delta (\mathcal{Z}x) \geq 1$. 
\end{enumerate}
Then, $h ,\eta , Q$ and $\mathcal{Z}$ have a  common unique  fixed point in $\mathcal{O}$.
\end{thm}
\begin{proof}  By Theorem  $2.2,$ we have   $d = hw = Qw = \mathcal{Z}x = \eta x.$ Thus,
$x$ becomes the  coincidence point of the pair $(\eta , \mathcal{Z} ).$
By  the weak compatibility of 
 the pair $(h, Q )$, we get
 $hd =hQw = Qhw = Qd.$
Thus, $d$ becomes a  coincidence point of the pair $(h, Q)$.
From $(ii)$, we get
$\gamma (Qd) \geq 1$ and  $\delta (\mathcal{Z}x) \geq 1$. 
So, $ \gamma (Qd)\delta (\mathcal{Z}x) \geq 1$ implies  \begin{equation}  \xi(\partial(hd , \eta x ) ) \leq \xi(s^{3}\partial(hd , \eta x )) \leq H(\xi (N_{s}(d , x) ), \omega(N_{s}(d , x)) ) \hspace{1.2cm} \end{equation}
\begin{align}
\mbox{Now, }  N_{s}(d, x) &= \max \{\partial(Qd, \mathcal{Z}x) , \partial(hd , Qd) , \partial(\mathcal{Z}x, \eta x),\frac{\partial(Qd, \eta x) + \partial(hd , \mathcal{Z} x)}{2s}\} \notag\\
&= \max  \{\partial(hd, \eta x) , \partial(hd , hd) , \partial(d, d),\frac{\partial(hd, \eta x) + \partial(hd , \eta x)}{2s}\}\notag\\
&\leq \max \{ \partial(hd , \eta x) , \partial(hd , \eta x) , 0 , \partial(hd , \eta x)\} 
 = \partial(hd , \eta x).\notag\\
\mbox{So, } N_{s}(d,x)&= \partial( hd , \eta x). \end{align}
By equations $(2.33)$, $(2.34)$ and the definition of H, we have
 $\xi(\partial(hd , \eta x )) $\\$\leq H(\xi (\partial(hd , \eta x )), \omega (\partial(hd , \eta x ))) \leq \xi(\partial(hd , \eta x ) ).$
  So, $H(\xi (\partial(hd , \eta x )), $\\$\omega (\partial(hd , \eta x ))) = \xi(\partial(hd , \eta x )),$ hence, $\xi (\partial(hd , \eta x )) = 0$ or $\omega (\partial(hd , \eta x )) = 0 .$
 Using the definition of $\xi$ and $\omega$, we have $ \partial(hd , \eta x ) = 0 $ implies $hd = \eta x. $ 
Also, $hd = Qd = d. $ Thus, $d$ becomes a fixed point of $h$ and $Q.$ By the weak compatibility of the pair  $(\eta, \mathcal{Z})$,  we get $\mathcal{Z}d  = \mathcal{Z}\eta x = \eta \mathcal{Z}x = \eta d.$ Hence, $d$ is a coincidence point of $(\mathcal{Z} ,\eta ).$ 
Again, by Theorem  $2.2,$ we have   $d = hw = Qw = \mathcal{Z}x = \eta x.$  Thus, $w$ is a coincidence point of $(h, Q).$
By $(ii)$,  we  have $\gamma (Qw) \geq 1 $ and $\delta(\mathcal{Z}d) \geq 1$.
So, $\gamma(Qw) \delta (\mathcal{Z}d) \geq 1$ implies 
\begin{align}\xi(\partial(hw , &\eta d ) ) \leq \xi (s^{3}\partial(hw , \eta d )) \leq H(\xi (N_{s}(w , d) ), \omega (N_{s}(w , d)) ) \\
\mbox{Now, } N_{s}(w &, d)) = \max \{\partial(Qw, \mathcal{Z}d) , \partial(hw , Qw) ,\partial(\mathcal{Z}d, \eta d), \hspace{5cm}\notag\\
&\hspace{4cm}\frac{1}{2s}(\partial(Qw, \eta d) + \partial(hw , \mathcal{Z}d))\}\notag\\
&\hspace{0.6cm}= \max  \{\partial(Qw, \mathcal{Z}d) , \partial(Qw , Qw) , \partial(\mathcal{Z}d, \mathcal{Z}d), \notag\\&\hspace{4cm}\frac{1}{2s}(\partial(Qw, \mathcal{Z}d) + \partial(Qw , \mathcal{Z}d))\} \notag\\
&\hspace{0.6cm} \leq \max \{ \partial(Qw , \mathcal{Z}d) , \partial(Qw , \mathcal{Z}d) , 0 ,  \partial(Qw , \mathcal{Z}d)\} \notag\\
&\hspace{0.6cm} = \partial(Qw , \mathcal{Z}d).\notag\\
\mbox{So, }  N_{s}(w , d)&= \partial( Qw , \mathcal{Z}d).\end{align}
From equations $(2.35)$, $(2.36)$ and by the definition of $H$, we have\\
 $\xi(\partial(hw , \eta d)) \leq H(\xi (\partial(hw , \eta d)), \omega (\partial( hw , \eta d)) \leq \xi(\partial(hw , \eta d ) ).$
 So, $H(\xi (\partial(hw , $\\$\eta d)), \omega (\partial( hw , \eta d)) = \xi(\partial(hw , \eta d)).$  Hence, $\xi(\partial(hw , \eta d )) = 0$ or $\omega (\partial(hw , \eta d )) = 0 .$ 
 Using the definition of $\xi$ and $\omega$, we get $\partial(hw , \eta d ) = 0.$ Thus, $hw = \eta d$.
Also, $hd = \mathcal{Z}d = d.$ Thus, $d$ becomes a common   fixed point  of $\mathcal{Z}$ and $\eta .$\\
Claim. $h, \eta, Q$ and  $\mathcal{Z}$ have a common   unique fixed point in $\mathcal{O}$. \\
Suppose $h, \eta , Q,$ and $\mathcal{Z}$ have a common fixed point $v$ (say)  in $\mathcal{O}.$
Thus,  $v$ becomes a coincidence point of $(h , Q)$. By $(ii)$, we have $\gamma(Qv)\geq 1$ and $\delta(\mathcal{Z}d) \geq 1.$ So, $\gamma(Qv) \delta (\mathcal{Z}d) \geq 1$ implies 
\begin{align} \xi (\partial (hv , &\eta d ) ) \leq \xi (s^{3}\partial(hv , \eta d )) \leq H(\xi (N_{s}(v , d) ), \omega (N_{s}(v , d)) ).\hspace{0.3cm}\\
\mbox{Now, }  N_{s}&(v , d)) = \max \{\partial(Qv, \mathcal{Z}d) , \partial(hd , Qd) ,\partial(\mathcal{Z}d,\eta d),\hspace{3.2cm}\notag\\&\hspace{4cm} \frac{1}{2s}(\partial(Qv, \eta d) + \partial(hv , \mathcal{Z}d))\}\notag\\
&\hspace{2.3cm}= \max \{\partial(v, d) , \partial(d , d) , \partial(d, d),\frac{1}{2s}(\partial(v, d) + \partial(v , d))\}\notag\\
&\hspace{2.3cm}\leq \max \{ \partial(v, d) , 0 , 0 ,  \partial(v,d )\}
= \partial(v , d).\notag\end{align}\begin{align}
\mbox{So, }  N_{s}(v, d)= \partial( v, d).\hspace{7cm}\end{align}
Thus, by inequalities $(2.37),(2.38)$ and the definition of $H ,  \xi,  \omega $, we have  $\xi(\partial(v , d ) )  \leq H(\xi (\partial(v , d) ), \omega (\partial(v, d)) ) \leq  \xi(\partial(v , d ) ).$ 
  So,   $\xi (\partial(v , d ) ) = 0$ or   \\ $\omega (\partial(v , d ) ) = 0.$ Hence,  $\partial(v , d ) = 0$ implies 
 $v =d.$ Thus $d$ is the   common unique fixed point of $h,\eta ,Q$ and $\mathcal{Z}$.\end{proof}
 The following are the consequences of Theorem $2.3$  to get the common unique fixed point for the pair $(h,Q)$ and $(\eta,\mathcal{Z})$.
\begin{cor}Let $C , D $ be two non empty closed subsets of a complete partial $b$-metric space $(\mathcal{O}, \partial , s)$ and $C \cap D  \neq \emptyset$. Assume that $h , \eta$ be two self maps on $C\cup D $  such that $hC \subseteq D$ and  $\eta D \subseteq C $. Suppose that there exist $\xi  \in \boldsymbol{\Xi} , \omega \in \boldsymbol{\Omega^{1}} , H \in \mathscr{C} $  such that $\xi(s^{3}\partial(hw, \eta v)) \leq  H(\xi(N_{s}(w, v) ,\omega(N_{s}(w,v)) )$ 
where $ N_{s}(w ,v) = \max\{\partial(w,v),\partial( hw, w) , \partial(v , \eta v ) , 
\frac{1}{2s}(\partial(w , \eta v) + \partial(hw , v)) \}.$ for all $w \in C$ and $v \in D.$
Then, $h$ and $\eta$ have a common unique fixed point in $C \cap D$.\end{cor}
\begin{proof}Define $\gamma , \delta : C \cup D \rightarrow \mathbb{R^{+}}$ as \vspace{0.2cm}\\$\gamma(w) = \begin{cases}
 1, & \text{if $w \in C$ } \\  
 0, & \text{otherwise} 
  \end{cases} $
 and  
 $\delta (w) =  \begin{cases}
 1, & \text{if $w \in D$ } \\  
 0, & \text{otherwise}
  \end{cases} $\vspace{0.2cm}\\
 For $w , v \in C \cup D$ with  $\gamma (w) \delta (v) \geq 1$, we have $\gamma (w) = 1 , \delta (v) =1$
  and   $w\in C ,$ \\ $v \in D. $ 
From given condition,  we have 
$\xi(s^{3}\partial(hw, \eta v)) \leq  H(\xi(N_{s}(w, v) ,$\\ $\omega(N_{s}(w,v) )  ).  $ 
Now, if $w \in C \cup D$ with $\gamma (w) \geq 1$, then $w \in C$ and $hw \in hC \subseteq D.$ So, $\delta (hw) \geq 1$.
Also, if $v \in C \cup D $  with $ \delta (v) \geq 1$ , then $v \in D $ and $\eta v \in \eta D \subseteq C.$ So,  $\gamma (\eta v) \geq 1 .$
Therefore, $(h , \eta )$ is cyclic $(\gamma , \delta)$-admissible map.
As  $C \cap D  \neq \emptyset$. So,  there exist $v_{0} \in C \cup D$ such that $\gamma (v_{0}) \geq 1$ and $\delta (v_{0}) \geq 1$. If $ \{ v_{n}\}_{n\in\textbf{W}}$ is a sequence in $ C \cup D$ with $v_{n} \rightarrow v$ and $\gamma (v_{n} ) \geq 1 , \delta (v_{n}) \geq 1 $,
for all $n$, then $v_{n} \in C$ and $v_{n} \in D .$
Since $C$ and $D$ are closed,  $v \in C$ and $ D$. Hence, $\gamma (v) \geq 1$  and $\delta (v) \geq 1$. 
Take $Q = \mathcal{Z} = I$ on $\mathcal{O}$. Thus, $(h,\eta )$ is generalized $TAC$-contractive map. Thus, $h$ and  $\eta $ have a common unique fixed point in $C \cap D$ by Theorem $2.3$.
\end{proof}
\begin{cor}
    Let $(\mathcal{O}, \partial )$ be a complete metric space. Assume that $h ,\eta , Q , \mathcal{Z}$ be four self maps on $\mathcal{O}$ such that;
 \begin{enumerate}[label=(\roman*)]
\item $h\mathcal{O} \subseteq \mathcal{Z}\mathcal{O}$ and $\eta \mathcal{O} \subseteq Q\mathcal{O}$, 
\item  $\xi (\partial(hw , \eta v)) \leq \xi(N_{s}(w ,v) ) - \omega(N_{s}(w,v) )  ,$
 where  $\xi \in \boldsymbol{\Xi} ,\omega \in \boldsymbol{\Omega^{1}} ,N_{s}(w,v) = \max\{\partial(Qw ,\mathcal{Z}v),\partial( hw, Qw) , \partial(\eta v , \mathcal{Z}v ),  \frac{\partial(Qw , \eta v) + \partial(hw, \mathcal{Z}v)}{2s} \} $, for all  \mbox{$w , v \in \mathcal{O}.$ }
\item  $(h, Q)$ and $(\eta, \mathcal{Z})$ are weakly compatible. 
\end{enumerate}
Then, $ h,\eta, Q$ and $\mathcal{Z} $ have a common  unique  fixed point in $\mathcal{O}$ assuming that atleast one of the ranges $h\mathcal{O} ,\eta \mathcal{O} ,Q\mathcal{O}$ and  $\mathcal{Z}\mathcal{O}$ is closed.
\end{cor}
\begin{proof}
    Define $\gamma , \delta : \mathcal{O} \rightarrow \mathbb{R^{+}}$ as 
 $\gamma (w) = 1$ and $\delta (w) = 1$  and 
define $H: \mathbb{R^{+}}\times \mathbb{R^{+}}\rightarrow \mathbb{R^{+}}$  as 
$H(t, v) = \begin{cases}
 t-v, & \text{if $t\geq v$ } \\  
 0, & \text{otherwise} 
  \end{cases}$ \\
then,  result follows from  Theorem  $2.3$.
\end{proof}
\begin{ex}
Let $ \mathcal{O} = \mathbb{R^{+}} $ and define $ \partial : \mathcal{O} \times \mathcal{O} \rightarrow \mathbb{R^{+}} $ as
$\partial(w,z) = \max \{w, z\} $, for all $ w , z \in \mathcal{O} $. Clearly,  it is a complete partial $b$-metric space with s=1.
Also, define $ \gamma , \delta  : \mathcal{O} \rightarrow \mathbb{R^{+}} $ and $ h, \eta , Q , \mathcal{Z} : \mathcal{O} \rightarrow \mathcal{O} $ as \vspace{0.3cm}\\
$\gamma(z) =\begin{cases}
 \frac{\sqrt{z}}{16\sqrt{2}}, & \text{if $ z \in (0, 32 )  $ } \vspace{0.2cm}\\  2, & \text{ otherwise  } 
  \end{cases}, ~ \delta(z) = \begin{cases}
 \frac{1}{4}, & \text{if $ z \in (0, 8 )  $ } \vspace{0.2cm}\\  
 1, & \text{otherwise } 
  \end{cases} $ 
  and  $ h(z) = z,$\vspace{0.3cm}\\
  $\eta(z) =  \begin{cases}
 0, & \text{if $ z \in [ 0, 64 ) $ } \vspace{0.2cm}\\  
 \frac{z}{2}, & \text{otherwise } 
  \end{cases} , ~\mathcal{Z}(z) = 
 z^{3} ,
  ~Q(z)=\frac{z^{2}}{2}, $ for all $z \in \mathcal{O}. $\vspace{0.3cm}\\
  Now, $\gamma(Qz)= \begin{cases}
 \frac{z}{32}, & \text{if $ z \in ( 0, 8 )  $ } \vspace{0.2cm}\\  
2, & \text{otherwise } 
  \end{cases} , ~\delta(\mathcal{Z}z)=\begin{cases}
 \frac{1}{4}, & \text{if $ z \in ( 0, 2 )  $ } \vspace{0.2cm}\\  
1, & \text{otherwise  } 
  \end{cases} $
 and \vspace{0.3cm}\\$\delta(hz) = \begin{cases}
 \frac{1}{4}, & \text{if $ z \in ( 0, 8 )  $ } \vspace{0.2cm}\\  
 1, & \text{otherwise } 
  \end{cases} ,~\gamma(\eta z)=2.$ \vspace{0.3cm}\\
 Since $0 \in \mathcal{O}$, $\gamma(Q(0)) = 2 \geq 1$  and $\delta (\mathcal{Z}(0)) = 1 \geq 1$.\\
 Also, if $\{v_{n}\}_{n \in W }$ is a sequence in $\mathcal{O}$ with $v_{n} \rightarrow v,\gamma(v_{n})\geq 1$ 
 and $\delta(v_{n})\geq 1$ for  all $n\in W,$ then $\gamma(v)\geq 1$ and $\delta(v) \geq 1$.\\
 Now, define $\xi , \omega: \mathbb{R^{+}} \rightarrow \mathbb{R^{+}}$ and 
 $H : \mathbb{R^{+}} \times  \mathbb{R^{+}} \rightarrow \mathbb{R^{+}} $ as 
 $\xi(t) = t ,\omega(t)= log(t+3)$ and 
 $H(s, t) = \frac{1}{2}s.$
 Clearly, $ \xi \in \Xi ,\omega \in \Omega^{1} ,  H \in \mathscr{C}.$\vspace{0.3cm}\\ Now, $\gamma(Qz) \delta(\mathcal{Z}w) = \begin{cases}
 \frac{z}{128}, & \text{if $z \in (0,8)$ and $w \in (0,2)$}\vspace{0.2cm}\\\frac{z}{32},  & \text{if $z \in (0,8)$ and $w \in \{0\} \cup [2, \infty) $}\vspace{0.2cm}\\
 \frac{1}{2}, & \text{if $z \in \{0\} \cup [8, \infty)$ and $w \in (0,2)$} \vspace{0.2cm}\\
2,  & \text{if $z \in \{0\} \cup [8, \infty)$ and $w \in   \{0\} \cup [2, \infty) $}
 \end{cases} $ \vspace{0.3cm}\\
 Note that if $\gamma(Qz) \delta(\mathcal{Z}w) \geq 1 $, then $z \in \{0\} \cup [8, \infty)$ and $ w \in \{0\}\cup [2,\infty ). $ \\
Now, if $z \in \{0\} \cup [8, \infty)$ and $ w \in \{0\}\cup [2,\infty)$, then
$\xi(s^{3}\partial(hz , \eta w)) = \xi(\partial(z , 0 )) = \xi(z) = z$ and $\frac{\partial(Qz, \mathcal{Z}w)}{2} = \frac{\partial(\frac{z^{2}}{2},w^{3})}{2} \leq \frac{N_{s}(z,w)}{2} = H(  \xi(N_{s}(z,w)) , $\\$\omega(N_{s}(z,w)) ).$
Thus, $\xi(s^{3}\partial(hz , \eta w)) \leq  H(  \xi(N_{s}(z,w)) , \omega(N_{s}(z,w)) ),$ for all  $z \in \{0\} \cup [8, \infty)$ and $w \in \{0\}\cup [2,\infty ), $ where $N_{s}(w, z ) = \max \{ \partial(Qw, \mathcal{Z}z) , $\\$\partial( hw, Qw) , \partial(\mathcal{Z}z , \eta z) ,$
 $\frac{\partial(Qw , \eta z) + \partial(hw, \mathcal{Z}z)}{2s} \}$. So, $(h,\eta)$  is generalized TAC-$(Q,\mathcal{Z})$ contractive map.
Note, if $\gamma(Qz) \geq 1$, then 
  $z\in \{0\}\cup [8,\infty ) $ implies $\delta(hz)\geq 1 $.
  Also, if $\delta(\mathcal{Z}z) \geq 1$, then 
  $z \in \{0\}\cup [2,\infty ) $ implies $\gamma(\eta z)\geq 1 $.
  Thus, $(h , \eta)$ is cyclic $(\gamma , \delta )$ admissible  with respect to $(Q , \mathcal{Z} ).$  Also, $h\mathcal{O} \subseteq \mathcal{Z}\mathcal{O} , \eta \mathcal{O} \subseteq Q\mathcal{O} $ and $h\mathcal{O}$ is closed.
Thus, by Theorem $2.2$ we get $\textbf{P}(h,Q)\neq \emptyset$ and $\textbf{P}(\eta ,\mathcal{Z})\neq \emptyset.$ Also, note that $0$ is the common fixed point of $h,Q,\eta,\mathcal{Z}$.
\end{ex}
\section{Statements and Declarations}
\noindent{\bf Ethical Approval}\\
Not applicable.\\

\noindent{\bf Competing interests}\\
The authors declare that they have no conflict of interest.\\

\noindent{\bf Authors' contributions}\\
All authors contributed equally to the manuscript.\\

\noindent{\bf Availability of data and materials}\\
No data or software were used to support the findings of this study, and no data or software has been generated.

\end{document}